\documentclass[reqn,9pt]{article}
\usepackage[flushleft]{threeparttable}
\usepackage[square,sort,comma,numbers]{natbib}
\usepackage{tikz}
\usepackage{graphicx}
\usepackage{amsmath, amssymb, amsthm}
\usepackage{latexsym}
\usepackage{rotate}
\usepackage{lscape}
\usepackage{color}
\usepackage{graphicx}

\newtheorem{theorem}{\bf Theorem}[section]

\newtheorem{corollary}{\bf Corollary}[section]
\newtheorem{remark}{\bf Remark}[section]
\newtheorem{definition}{\bf Definition}[section]
\newtheorem{example}{\bf Example }[section]
\begin{document}
\title{\bf Further Results Involving Residual and Past Extropy with their Applications}
\author{Mohamed Kayid$^{\dag,}$\thanks{Corresponding author, Email address: drkayid@ksu.edu.sa}\\
%EndAName
~~\\
{\small Department of Statistics and Operations Research}\\
{\small College of Science, King Saud University, P.O. Box 2455, Riyadh 11451, Saudi Arabia}\\
{\small ~~}\\
}
\date{}
\maketitle

\begin{abstract}
%\boldmath
In this paper we focus on the study of the monotonicity properties of the residual and the past extropy as well as on some characterization problems. We then apply the derived results to analyze further stochastic aspects of order statistics, coherent systems, and record values. Using various examples, we illustrate the applicability and significance of the results obtained. In addition, nonparametric estimators for the residual and past extropy measures are introduced. The performance of the acquired estimators has been illustrated using simulated data sets and also real data sets.

\end{abstract}

\vskip 2mm
%\noindent {\bf Mathematics Subject Classification:}

\section{Introduction}\label{intro}

Let $X$ be a non-negative random variable (RV) characterized by an absolutely continuous density function (DF) called $f(x)$. The Shannon differential entropy (Shannon \cite{Shannon1948amathematical}) is defined as $\mathcal{H}(X)=-\mathbb{E}[\log f(X)],$ provided that the expectation exists. A recent study by Lad \emph{et al.} \cite{lad2015extropy} introduced a new uncertainty metric known as extropy, which serves as a dual complement to entropy.

Let $X$ be an observable quantity with a finite discrete range of possible values $\{x_1,\ldots,x_N\}$ with probability vector $\textbf{p}_N=(p_1,\ldots,p_N),$ where $p_i=P(X=x_i),\ i=1,\ldots, N$. As pointed out by Lad \emph{et al.} \cite{lad2015extropy}, the extropy is originally defined for a discrete quantity as $J(\mathbf{p}_N)=-\sum_{i=1}^{N}(1-p_i)\log(1-p_i)$. However, if the range of possible values of an RV is extremely refined, as is the case for a continuous RV with DF $f$, then the extropy measure can be approximated by
\begin{equation}\label{extropy}
\mathfrak{J}(X)=-\frac{1}{2}\int_{0}^{\infty} f^2(x) dx,
\end{equation}
provided that the integral exists. Using the extropy in this integral form would only be suitable as an approximation to a truly discrete observation sequence if the observation units are very refined.

The measure of extropy $\mathfrak{J}(X)$ is commonly used to quantify the uncertainty surrounding the lifetime $X$ of a fresh unit. However, in certain scenarios, operators have information about the current age of the system. For example, they know that the system is operational at time $t$ and want to assess the uncertainty surrounding the remaining lifetime $X_t = [X - t \mid X > t]$. In such cases, the conventional extropy $\mathfrak{J}(X)$ is insufficient. In their seminal work, Qiu and Jia \cite{qiu2018residual} introduced a more dynamic version of the extropy of a life span with respect to its age, namely the residual extropy (REX). The REX is defined as:
\begin{eqnarray}
\mathfrak{J}(X;t)&=&-\frac{1}{2}\int_{0}^{\infty}f^2(x;t)dx
=-\frac{1}{2}\int_{0}^{\infty}\left(\frac{f(x+t)}{S(t)}\right)^2dx
=-\frac{1}{2}\int_{t}^{\infty}\left(\frac{f(u)}{S(t)}\right)^2du,\label{def:extropy:time}
\end{eqnarray}
where $f(x;t)={f(x+t)}/{S(t)}$ is the DF of $X_t$ and $S(t)$ is the survival function of $X.$ It is worth noting that the REX has values in the range of $[-\infty,0)$ and corresponds to the extropy of $[X|X>t]$.

From \eqref{def:extropy:time}, the extropy of $X_t$ can be reformulated as
\begin{eqnarray}\label{eq:rex:2}
 \mathfrak{J}(X;t) &=& -\frac{1}{2}\int_{t}^{\infty}\left(\frac{f(u)}{S(t)}\right)^2du= -\frac{1}{4}\int_{t}^{\infty}2\frac{f(u)}{S(u)}\frac{f(u)S(u)}{S(t)S(t)}du\nonumber\\
 &=& -\frac{1}{4}\int_{t}^{\infty}\lambda(u)f_{12}(u|t)du=-\frac{1}{4}\mathbb{E}[\lambda(X_{12})|X_{12}>t],
\end{eqnarray}
where $\lambda(u)=f(u)/S(u)$ represents the hazard rate function of $X$ and $f_{12}(u|t)$ is the DF of $[X_{12}|X_{12}>t].$

Given the duality, it is reasonable to explore the extropy of the inactivity time denoted as $X_{[t]} = [t-X| X \leq t]$. At age $t$, the DF of the past lifetime, $X_{[t]}$, can be expressed as $f(x;[t])=f(t-x)/F(t)$, where $0<x\leq t$ and $F(t)$ is the cumulative distribution function (CDF) of $X.$ Consequently, the past extropy (PEX) is defined as follows (see \cite{sathar2020past}):
\begin{eqnarray}\label{eq:pex}
\widetilde{ \mathfrak{J}}(X;t) =-\frac{1}{2}\int_0^{\infty} f^2(x;[t])  dx=-\frac{1}{2}\int_{0}^t \left[\frac{f(t-x)}{{F}(t)}\right]^2dx=-\frac{1}{2}\int_{0}^t \left[\frac{f(u)}{{F}(t)}\right]^2du.
\end{eqnarray}
Similar to the definition of extropy in equation (\ref{def:extropy:time}), the PEX also falls within the range of $[-\infty,0)$ and corresponds to the extropy of $[X|X\leq t]$. Furthermore, utilizing equation (\ref{eq:pex}), we can provide an alternative expression for the PEX, which is given by

\begin{eqnarray}\label{eq:pex:3}
\widetilde{\mathfrak{J}}(X;t) &=& -\frac{1}{2}\int_{t}^{\infty}\left(\frac{f(u)}{F(t)}\right)^2du= -\frac{1}{4}\int_{t}^{\infty}2\frac{f(u)}{F(u)}\frac{f(u)F(u)}{F(t)F(t)}du\nonumber\\
 &=& -\frac{1}{4}\int_{t}^{\infty}\tau(u)f_{22}(u|t)du=-\frac{1}{4}\mathbb{E}[\tau(X_{22})|X_{22}\leq t].
\end{eqnarray}
where $\tau(u)=f(u)/F(u)$ represents the reversed hazard rate function of $X$ and $f_{22}(u|t)$ is the DF of $[X_{22}|X_{22}\leq t].$

The subsequent two concepts are closely connected to extropies concerning residual life and inactivity time, respectively.
\begin{definition}\rm
Suppose $X$ is a nonnegative RV characterized by an absolutely continuous DF denoted as $f(x)$.
\begin{itemize}
  \item We say that $X$ has decreasing (increasing) residual extropy (DREX(IREX)) if $\mathfrak{J}(X;t)$ is decreasing (increasing) in $t>0$.
  \item We say that $X$ has increasing past extropy (IPEX) if $\widetilde{\mathfrak{J}}(X;t)$ is increasing in $t>0$.
\end{itemize}
\end{definition}

The concept of extropy, which measures information, has been extensively explored by researchers such as Lad \emph{et al.} \cite{lad2015extropy}, Qiu \cite{qiu2017extropy}, and Qiu and Jia \cite{qiu2018extropy, qiu2018residual} and the references therein. Qiu \cite{qiu2017extropy} has conducted a comprehensive study on extropy and has engaged in some comparative studies. The conditions that lead to the uniqueness of extropy in terms of order statistics and record values were also investigated by Qiu \cite{qiu2017extropy}. In a related work,  Shrahili and Kayid \cite{shrahili2023excess} introduced the concept of REX for the ith-order statistic. Their efforts yielded several valuable results. Furthermore, Kayid and Alshehri \cite{kayid2023system} recently explored the PEX of coherent systems. Specifically, they considered a scenario where all system components are assumed to be inactive after a given time. This paper aims to extend the study of REX and PEX by considering their dynamic versions in situations where uncertainty is sought in a given time frame. We develop our findings in the context of order statistics, coherent systems and record values.

This study also aims to fill the gap in the literature regarding past extropy properties of order statistics, coherent systems, and record values. Using extropy is advantageous as it is computationally simpler than other information measures. This allows for efficient calculation of REX and PEX for order statistics, coherent systems, and record values. The proposed method demonstrates the versatility and applicability in reliability and distribution analysis and provides researchers and practitioners with valuable insights for informed decision making. The extropy measure and its dynamical versions have found applicability in the field of practical reliability and pattern recognition (see, e.g.,  Balakrishnan et al. \cite{balakrishnan2022tsallis}, Chakraborty and Pradhan \cite{chakraborty2023cumulative}, Chakraborty et al. \cite{chakraborty2023weighted} and the reference therein). By leveraging the insights and techniques from these publications, researchers and practitioners can further enhance their models and algorithms in these domains.
\medskip

This paper is organized as follows: In Section 2, we delve into an overview of coherent systems and explore $k$-record values. Section 3 unveils novel findings concerning the REX of an RV. We then apply these results to analyze order statistics, coherent systems, and record values. In Section 4, we provide the same novel findings concerning the PEX. Nonparametric estimators for the proposed extropy measures are introduced, and their performance is illustrated using simulated and real data sets in Section 5. Finally, in Section 6, we summarize our contributions and propose avenues for future research.

\section{Preliminaries}
In this section, we provide a brief overview of coherent systems and $k$-record values, which will play a crucial role in the subsequent analysis and developments of this paper.

\subsection{Coherent systems}

In the context of reliability theory, a coherent system is characterized by the absence of irrelevant components and a monotone structure function. The definition and fundamental properties of coherent systems can be found in references such as Barlow and Proschan \cite{barlow1975statistical}. An example of a coherent system is the $(n-i+1)$-out-of-$n$ system, which functions properly as long as at least $i$ components are operational. Assume a coherent system comprising $n$ independent and identically distributed (i.i.d.) components with absolutely continuous lifetimes denoted as $X_1, \cdots, X_n$, all following a common CDF $F$ and the survival function $S(x).$ The lifetime of the coherent system can be represented by $T$. Additionally, let $X_{1:n} < \cdots < X_{n:n}$ represent the order statistics of the component lifetimes. The system's signature is defined as the vector $\textbf{s}=(s_1,\cdots,s_n)$, where $s_i = P(T=X_{i:n})$ for $1\leq i\leq n$, indicating the probability that the $i$-th failure results in system failure. It has been proven that (see e.g., in Samaniego \cite{samaniego2007system}):
\begin{eqnarray}
S_T(t)&=&\sum_{i=1}^n s_iS_{i:n}(t),\label{def:survival:T}\\
F_T(t)&=&\sum_{i=1}^n s_iF_{i:n}(t),\label{def:CDF:T}
\end{eqnarray}
where
\begin{eqnarray}
S_{i:n}(t)&=&\sum_{k=0}^{i-1}{n\choose k}F^k(t)S^{n-k}(t),\label{def:survival:Xin}\\
F_{i:n}(t)&=&\sum_{k=i}^{n}{n\choose k}F^k(t)S^{n-k}(t),\label{def:CDF:Xin}
\end{eqnarray}
is the survival function and distribution function of $X_{i:n}.$ Consequently, we have the following expression:
\begin{equation*}\label{density:T}
f_T(t)=\sum_{i=1}^n s_i f_{i:n}(t),
\end{equation*}
where
\begin{equation*}\label{def:orderstatistics}
f_{i:n}(t)=i{n\choose i}[F(t)]^{i-1}[S(t)]^{n-i}f(t),\ t>0,
\end{equation*}
which represents the DF of $X_{i:n}$ (see David and Nagaraja \cite{david2004order}). Equation $\eqref{def:survival:T}$ indicates that the distribution of the coherent system's lifetime is solely determined by the system's design through its corresponding signature. Further details can be found in Samaniego \cite{samaniego2007system} and the references cited therein.

It is worth noting that the hazard rate function of $T$ can be expressed as
\begin{equation}\label{hazard:system}
\lambda_{T}(t)=\frac{f_T(t)}{S_T(t)}=\frac{\sum_{i=0}^{n-1}(n-i)s_{i+1}{n\choose i}(\phi(t))^{i}}{\sum_{i=0}^{n-1}\left(\sum_{j=i+1}^{n}s_j\right){n\choose
i}(\phi(t))^i}\lambda(t),
\end{equation}
where $\phi(t)=\frac{F(t)}{S(t)}$ is an increasing function of $t\in (0,\infty)$ and $\lambda(t)=\frac{f(t)}{S(t)}$ represents the hazard rate function of $X$. It is clear that the function $\lambda_{T}(t)/\lambda(t)$ is increasing in $t$ if the rational function
\begin{equation}\label{rational}
\psi(x)=\frac{\sum_{i=0}^{n-1}(n-i)s_{i+1}{n\choose i}x^i}{\sum_{i=0}^{n-1}\left(\sum_{j=i+1}^{n}s_j\right){n\choose
i}x^i},
\end{equation}
is increasing with respect to $x\in (0,\infty).$

%===============================================================
\subsection{k-Records}
%===============================================================

Let us introduce some preliminary concepts related to $k$-record values. Consider a sequence of i.i.d. RVs denoted by $\{X_i, i\geq 1\}$, with CDF $F(x)$ and DF $f(x)$. An observation $X_j$ is referred to as an upper record value if it is greater than $X_i$ for every $j>i$. To quantify these upper record values, Dziubdziela and Kopocinski \cite{dziubdziela1976limiting} introduced the indices $\{R_k(n), n\geq 1\}$, which represent the times of the $n$-th upper $k$-record for the sequence $\{X_i, i\geq 1\}$. These indices are defined by
\[
R_k(1)=1,\ R_k(n+1)=\min\{j:j>R_k(n),X_{j:j+k-1}>X_{R_k(n):R_k(n)+k-1}\},
\]
where $X_{j:m}$ denotes the $j$-th order statistic derived from a set of i.i.d. RVs of size $m$. With this in mind, we can define $U_{n(k)}$ as a sequence of $n$-th upper $k$-record values from the sequence $\{X_i, i\geq 1\}$, which can be expressed as $U_{n(k)}=X_{R_k(n):R_k(n)+k-1}$. Consequently, we have
\begin{equation}\label{PDFkth}
f_{n(k)}(x)=\frac{k^{n}}{\Gamma(n)}[\overline{F}(x)]^{k-1}[-\log \overline{F}(x)]^{n-1}f(x),\ x>0,
\end{equation}
\begin{equation}\label{suv:Xn+1}
\overline{F}_{n(k)}(x) =  [\overline{F}(x)]^k\, \sum_{i=0}^{n-1} \frac{[-k\log \overline{F}(x)]^i}{i!}=\frac{\Gamma(n,-k\log \overline{F}(x))}{\Gamma(n)}, \;\; x \geq 0,
\end{equation}
where
\begin{equation*}\label{gamma}
\Gamma(a,x)=\int_{x}^{\infty}u^{a-1}e^{-u}du,\ a,x>0,
\end{equation*}
is the incomplete gamma function and $\Gamma(n)=\Gamma(n,0)$ is the complete gamma function. Based on \eqref{PDFkth} and \eqref{suv:Xn+1}, we can express the hazard rate function of $U_{n(k)}$ as
\begin{equation}\label{hazrd:krecord}
\lambda_{n(k)}(t)=\frac{f_{n(k)}(t)}{\overline{F}_{n(k)}(t)}
=\frac{k^n\varphi^{n-1}(t)}{\Gamma(n)\sum_{i=0}^{n-1} \frac{(k\varphi(t))^i}{i!}}\lambda(t),
\end{equation}
where $\varphi(t)=-\log\overline{F}(t)$, which is an increasing function of $t\in (0,\infty).$

%-------------------------------------------------------------------------------------
\section{Results on Residual Extropy}
In reliability theory, stochastic comparisons between the RV $X$ and its residual lifetime are commonly employed to study aging properties. Two well-known properties are the new better than used (NBU) property and the decreasing mean residual life (DMRL) property. The NBU property holds if, for all $t>0$, the RV $X$ stochastically dominates $X_t$ in the usual stochastic order. On the other hand, the DMRL property holds if, for all $t\geq0$, $X$ stochastically dominates $X_t$ in terms of the mean residual life order. Detailed explanations and definitions of stochastic orders can be found in Shaked and Shanthikumar \cite{shaked2007stochastic}. The forthcoming theorem provides a characterization of the DREX property.
\begin{theorem}\label{thm:Tsallis}
An RV $X$ with DF $f$ is DREX if and only if
\begin{equation}\label{extropy:equi}
\mathfrak{J}(X_s;t)\leq \mathfrak{J}(X;t),
\end{equation}
for all $s,t>0.$
\end{theorem}
\begin{proof}
Recall that $X$ is DREX if and only if
\begin{equation}\label{extropy:equi1}
\mathfrak{J}(X;s+t)\leq \mathfrak{J}(X;t),\ \text{for all}\ s,t>0.
\end{equation}
By noting that
\[
S(x;s)=\frac{S(x+s)}{S(s)},\ \ f(x;s)=\frac{f(x+s)}{S(s)},
\]
for all $s,t\geq0,$ we have,
\begin{eqnarray*}
\mathfrak{J}(X_s;t)&=&-\frac{1}{2}\int_{t}^{\infty}\left[\frac{f(x;s)}{S(x;s)}\right]^2dx\nonumber\\
&=&-\frac{1}{2}\int_{t}^{\infty}\left[\frac{f(x+s)}{S(t+s)}\right]^2dx\ (\text{by taking u=x+s})\nonumber\\
&=&-\frac{1}{2}\int_{s+t}^{\infty}\left[\frac{f(u)}{S(t+s)}\right]^2du.\label{JXst}
\end{eqnarray*}
The conclusion drawn from the previous step is that $\mathfrak{J}(X_s;t)=\mathfrak{J}(X;s+t)$. This establishes the equivalence between \eqref{extropy:equi1} and \eqref{extropy:equi}. Therefore, the proof is now complete.
\end{proof}
The overall characterization problem aims to determine the unique identification of a distribution function based on the REX. Toomaj \emph{et al.} \cite{toomaj2023extropy} have shown that when the DF $f(x)$ is decreasing in $x$, the distribution function $X$ can be uniquely identified by its REX function $\mathfrak{J}(X;t)$. Similar conclusions can be obtained based on the monotonicity properties of the REX function. This area of research holds significant potential for further exploration, as evidenced by the diverse range of results and characterizations presented in the existing literature.
\begin{theorem}
If $X$ has an absolutely continuous CDF $F(t)$ and satisfies the IREX property, then $\mathfrak{J}(X;t)$ uniquely determines $F(t).$
\end{theorem}
\begin{proof}
It is evident that
\[
\mathfrak{J}'(X;t)=\frac{\lambda^2(t)}{2}-\lambda(t)\mathfrak{J}(X;t).
\]
If we substitute $x=\lambda(t)$, then, for a fixed $t>0$, $\lambda(t)$ becomes a positive solution of the following equation
\begin{equation}\label{derviative}
\xi_t(x)=\frac{x^2}{2}-x\mathfrak{J}(X;t)-\mathfrak{J}'(X;t).
\end{equation}
Furthermore, the increasing nature of $\mathfrak{J}(X;t)$ implies that $\xi_t(0)=-\mathfrak{J}'(X;0)\leq0$ and $\xi_t(\infty)=\infty$. On the other hand, we can deduce that
\[
\xi_t'(x)=x-\mathfrak{J}(X;t)\ \text{and}\ \xi_t''(x)=1>0.
\]
The function $\xi_t(x)$ exhibits a pattern of initially decreasing and then increasing with respect to $x$, reaching a minimum at $x_t=\mathfrak{J}(X;t)$. This implies that equation \eqref{derviative} possesses a unique positive solution $\lambda(t)$ for all $t$. As a result, $\mathfrak{J}(X;t)$ uniquely determines $\lambda(t)$ and consequently the distribution function $F$.
\end{proof}

Consider another continuous RV $Y$ with DF $f_Y(x)$, CDF $F_Y(x)$, survival function $S_Y(x)$, hazard rate $\lambda_Y(x)$, and reversed hazard rate $\tau_Y(x)$. In the context of stochastic orders, we say that RV $X$ is smaller than $Y$ in the likelihood ratio order, denoted as $X\leq_{lr}Y$ if the likelihood ratio $f_Y(x)/f_X(x)$ is increasing in $x$ over the union of their supports. For further information on stochastic orders, interested readers can refer to the work of Shaked and Shanthikumar \cite{shaked2007stochastic}. Now, we have the following theorem.
\begin{theorem}\label{thm:DREX}
Assume that $X\leq_{lr}Y$ and $\lambda_Y(t)/\lambda_X(t)$ is increasing for $t\geq0$. If $X$ has the DREX property, then it follows that $Y$ also possesses the DREX property.
\end{theorem}
\begin{proof}
Let $\eta(t)=\lambda_Y(t)/\lambda_X(t)$. Based on \eqref{eq:rex:2}, we have the following relationship
\begin{eqnarray*}
  -4\mathfrak{J}(Y;t)&=&\mathbb{E}[\lambda_Y(Y_{12})|Y_{12}>t] \\
  &=&\mathbb{E}[\eta(Y_{12})\lambda_X(Y_{12})|Y_{12}>t].
\end{eqnarray*}
Therefore, the goal is to prove that $\beta(t)=\mathbb{E}[\eta(Y_{12})\lambda_X(Y_{12})|Y_{12}>t]$ is an increasing function of $t\geq0$. By performing some manipulation, we obtain the following expression:
\begin{eqnarray*}
  \beta'(t)&=&\left(\frac{\int_{t}^{\infty}\eta(y)\lambda_X(y)2f_Y(y)S_Y(y)dy}
  {S_Y^2(t)}\right)' \\
  &=&\frac{-\eta(t)\lambda_X(t)2f_Y(t)S^3_Y(t)+2f_Y(t)S_Y(t)\int_{t}^{\infty}\eta(y)
  \lambda_X(y)2f_Y(y)S_Y(y)dy}{S_Y^4(t)} \\
  &=&2\lambda_Y(t)[\beta(t)-\eta(t)\lambda_X(t)].
\end{eqnarray*}
So, we need to show that $\beta'(t)\geq0$ for all $t>0.$ In order to demonstrate that $\beta(t)\geq\eta(t)\lambda_X(t)$ for all $t\geq0$, it is sufficient to establish the following inequality for all $t\geq0$:
\begin{eqnarray*}
  &&\beta(t)-\eta(t)\lambda_X(t)\\
  &=&\int_{t}^{\infty}\frac{2f_Y(y)S_Y(y)}
  {S_Y^2(t)}[\eta(y)\lambda_X(y)-\eta(t)\lambda_X(t)]dy \\
  &=&\int_{t}^{\infty}\frac{2f_Y(y)S_Y(y)}
  {S_Y^2(t)}\eta(y)[\lambda_X(y)-\lambda_X(t)]dy \\
  &+&\int_{t}^{\infty}\frac{2f_Y(y)S_Y(y)}
  {S_Y^2(t)}\lambda_X(t)[\eta(y)-\eta(t)]dy\\
  &=&A_1(t)+A_2(t)\geq0.
\end{eqnarray*}
Let us denote $\alpha(t)=\mathbb{E}[\lambda_X(X_{12})|X_{12}>t]$. Based on the assumption that $\mathfrak{J}(X;t)$ is decreasing for $t\geq0$, we can equivalently conclude that $\alpha(t)$ is an increasing function for all $t\geq0$. It is worth noting that $\alpha'(t)=2\lambda_X(t)(\alpha(t)-\lambda_X(t))$, and we have $\alpha(t)\geq\lambda_X(t)$ for all $t\geq0$. In other words, we have the following inequality
\begin{equation}\label{extropy:11}
\int_{t}^{\infty}2f_X(x)S_X(x)[\lambda_X(x)-\lambda_X(t)]dx\geq0,
\end{equation}
for all $t\geq0$. Moreover, since $X\leq_{lr}Y$, it means that the ratio $f_Y(y)/f_X(y)$ increases with the increase of $y\geq0$. Now, we can readily derive the following inequality for all $t\geq0$ by utilizing Lemma 7.1(i) of \cite{barlow1975statistical} on \eqref{extropy:11}:

\begin{eqnarray*}
  A_1(t)&=&\frac{1}{S_Y^2(t)}\int_{t}^{\infty}\frac{2f_Y(y)S_Y(y)\eta(y)}
  {2f_X(y)S_X(y)}2f_X(y)S_X(y)[\lambda_X(y)-\lambda_X(t)]dy \\
  &=&\frac{1}{S_Y^2(t)}\int_{t}^{\infty}\left(\frac{f_Y(y)}{f_X(y)}\right)^2
  2f_X(y)S_X(y)[\lambda_X(y)-\lambda_X(t)]dy\geq0.
\end{eqnarray*}

Since $\eta(t)$ is nonnegative and increasing in $t\geq0$, we can establish the following inequality for all $t\geq0$:
\begin{equation*}
  A_2(t)=\int_{t}^{\infty}\frac{2f_Y(y)S_Y(y)}
  {S_Y^2(t)}\lambda_X(t)[\eta(y)-\eta(t)]dy\geq0.
\end{equation*}
Therefore, we can obtain the inequality $\beta(t)\geq\eta(t)\lambda_X(t)$ for all $t\geq0$ and this completes the proof.
\end{proof}
In their work, Toomaj \emph{et al.} \cite{toomaj2023extropy} demonstrated that the property of decreasing uncertainty in a residual lifetime, as measured by REX, is preserved when forming parallel systems. The preservation of the DREX property under the formation of series systems is a question of interest. However, we provide a counterexample to demonstrate that DREX is not preserved under such system formations.
\begin{example}\label{exam:piecewise}\rm
Consider an RV $X$ with a survival function as follows:
\begin{equation}
S(x) =
    \begin{cases}
        1-\frac{x^2}{2}, & \text{if}\ 0\leq x<1\\
        \frac{2}{3}-\frac{x^2}{6}, & \text{if}\ 1\leq x< 2\\
        0, & \text{if}\ x\geq2
    \end{cases}.
\end{equation}
The REX of $X$ can be easily observed as
\begin{equation}
\mathfrak{J}(X;t) =
    \begin{cases}
        \frac{2}{27}\frac{9t^3-16}{(t^2-2)^2}, & \text{if}\ 0\leq t<1\\
        \frac{2}{3}\frac{t^2+2t+4}{(t+2)(t^2-4)}, & \text{if}\ 1\leq t< 2\\
    \end{cases}.
\end{equation}
\begin{figure}
\centering
\includegraphics[width=0.43\textwidth]{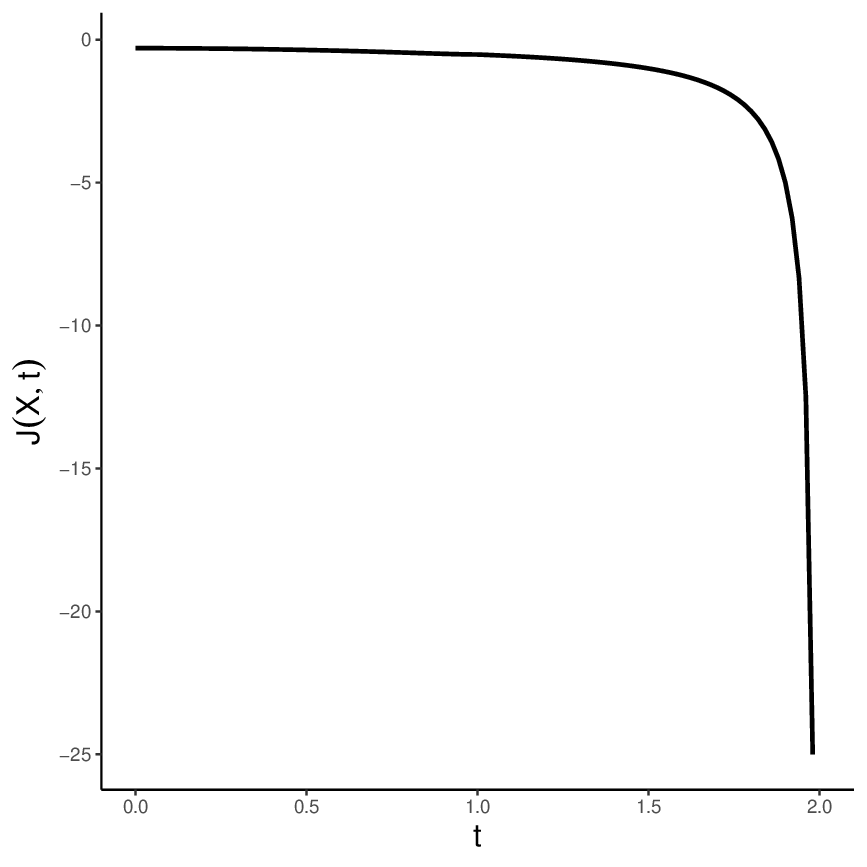}
\includegraphics[width=0.43\textwidth]{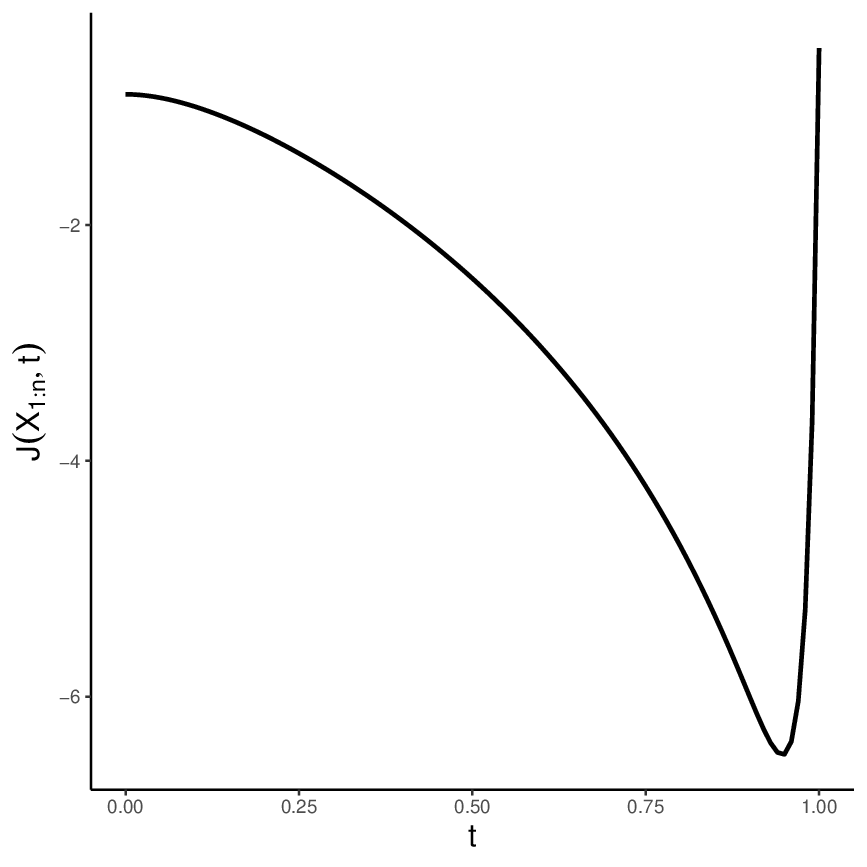}
\caption{Exact values of $\mathfrak{J}(X;t)$ (left panel) and $\mathfrak{J}(X_{1:n};t)$ (right panel) for the example provided in Example \ref{exam:piecewise}, as a function of $t$.} \label{fig:exp}
\end{figure}
By examining Figure \ref{fig:exp}, it becomes evident that $\mathfrak{J}(X;t)$ is decreasing in $t,$ indicating that RV $X$ possesses the DREX property. To analyze the relationship between $\mathfrak{J}(X_{1:n};t)$ and the time $t$, numerical methods are employed due to the challenge of deriving an explicit expression. Specifically, we consider the case of $0<t<1$ and plot $\mathfrak{J}(X_{1:n};t)$ for $n=15,$ as shown in Figure \ref{fig:exp}. Notably, it is evident that $\mathfrak{J}(X_{1:n};t)$ does not exhibit a monotonic behavior concerning $t$. Consequently, we can conclude that $X_{1:n}$ does not possess the DREX property.
\end{example}
\begin{remark}\rm
In Example \ref{exam:piecewise}, it is important to note that the ratio $\lambda_{1:n}(t)/\lambda(t)=n$ exhibits both increasing and decreasing behavior, while the ratio $f_{1:n}(t)/f(x)=nS^{n-1}(t)$ is decreasing with respect to $t$. This implies that $X\geq_{lr}X_{1:n}$. We have determined that the condition $X\leq_{lr}Y$ is essential in Theorem \ref{thm:DREX} and cannot be omitted. In the proof of Theorem 5.4 by Toomaj \emph{et al.} \cite{toomaj2023extropy}, recall that the condition $X\leq_{lr}Y$ in Theorem \ref{thm:DREX} can be loosened to $X\leq_{hr}Y$ (the hazard rate order) by including the condition $\lim_{t\rightarrow\infty}S_Y(t)/S_X(t)<\infty$. It should be noted that this additional condition may not always hold in practice.
\end{remark}
Now, we provide another useful theorem which is given in the next theorem.
\begin{theorem}
If $X_{i:n}$ is DREX, then $X_{i+1:n}$ , $X_{i:n-1}$  and $X_{i+1:n+1}$  are also DREX.
\end{theorem}
\begin{proof}
Let $\lambda_{i_1:n_1}(t)$ and $\lambda_{i_2:n_2}(t)$ be the hazard rate functions of order statistics $X_{i_1:n_1}$ and $X_{i_2:n_2},$ respectively. From \eqref{def:survival:Xin}, we have
\begin{equation}\label{eq:rationhazard}
\frac{\lambda_{i_2:n_2}(t)}{\lambda_{i_1:n_1}(t)}\propto\left(\frac{F(t)}{S(t)}\right)^{i_2-i_1}
\frac{\sum_{k=0}^{i_1-1}{n_1\choose k}(F(t)/S(t))^k}{\sum_{k=0}^{i_2-1}{n_2\choose k}(F(t)/S(t))^k}
\end{equation}
Due to Nagaraja \cite{nagaraja1990some}, we can observe that equation \eqref{eq:rationhazard} exhibits an increasing trend in $t$ under the following three scenarios: (i) when $n_1=n_2=n$, $i_1=i$, and $i_2=i+1$; (ii) when $n_1=n$, $n_2=n-1$, and $i_1=i_2=i$; (iii) when $n_1=n$, $n_2=n+1$, $i_1=i$, and $i_2=i+1$. Additionally, it is not hard to see that $X_{i:n}\leq_{lr}X_{i+1:n}$, $X_{i:n}\leq_{lr}X_{i:n-1}$, and $X_{i:n}\leq_{lr}X_{i+1:n+1}$. Consequently, by invoking Theorem \ref{thm:DREX}, we can readily establish the claim.
\end{proof}
As specified before, Toomaj \emph{et al.} \cite{asadi2000residual} demonstrated that the property of decreasing uncertainty in a residual lifetime, as measured by REX, is preserved when forming parallel systems. We can now extend this result to the coherent systems.
\begin{theorem}\label{thm:coherent}
Consider a coherent system with signature $\textbf{s}=(s_1,\cdots,s_n)$ comprising $n$ i.i.d. component lifetimes drawn from a CDF denoted by $F$. Let $T$ represent the lifetime of this system. We assume that $s_1\leq s_2\leq \ldots \leq s_n$, and the function $\psi(x)$ defined in \eqref{rational} is an increasing function for $x>0$. If $X$ exhibits the DREX property, then $T$ also possesses the DREX property.
\end{theorem}
\begin{proof}
Let us assume that RV $X$ possesses the DREX property. Considering the assumption $s_1\leq s_2\leq \ldots \leq s_n$, we can establish that $X\leq_{lr}T$ based on Lemma 2.1 of \cite{toomaj2015comparisons}. Additionally, since the function $\psi(x)$ defined in \eqref{rational} is increasing in $x$, we can infer that $\lambda_{T}(t)/\lambda_X(t)$ is increasing in $t$. Therefore, the assumption of Theorem \ref{thm:DREX} is satisfied, enabling us to conclude that $T$ also possesses the DREX property.
\end{proof}
An immediate consequence of Theorem \ref{thm:coherent} can be seen in the following.
\begin{corollary}
If $X$ is DREX, then $X_{n:n}$ is DREX.
\end{corollary}
\begin{proof}
The signature of the parallel system is $\mathbf{s}=(0,0,\ldots,0,1)$, which satisfies the assumption of Theorem \ref{thm:coherent}. Furthermore, it is evident that the function
\begin{equation*}
\psi(x)=\frac{nx^{n-1}}{\sum_{i=0}^{n-1}{n\choose i}x^i},
\end{equation*}
exhibits a monotonically increasing behavior with respect to $x$. This establishes the conclusion of the theorem.
\end{proof}
\begin{example}\rm
Consider the lifetime $T$ of a coherent system with signature $\textbf{s}=(0,0,\frac{1}{4},\frac{3}{4})$, consisting of $n=4$ i.i.d. components with a common CDF $F$. The system signature satisfies the assumption of Theorem \ref{thm:coherent}. Furthermore, by examining \eqref{rational}, we can observe that
\begin{equation*}
\psi(x)=\frac{3x^3+3x^2}{3x^3+6x^2+4x+1}.
\end{equation*}
We can easily verify that $\psi'(x)>0$ for all $x>0$. Therefore, by assuming the DREX property of the component lifetime, we can conclude that the lifetime $T$ of the coherent system is also DREX, as stated in Theorem \ref{thm:coherent}.
\end{example}
Now, we investigate the DREX properties of $k$-record values.
\begin{theorem}
If $X$ is DREX, then $U_{n(1)}$ is also DREX.
\end{theorem}
\begin{proof}
By examining \eqref{hazrd:krecord}, it becomes evident that the function $\lambda_{n(k)}(t)/\lambda(t)$ displays a monotonic increase concerning $t$. Additionally, based on \eqref{PDFkth}, the function
\begin{equation*}
\frac{f_{n(k)}(x)}{f(x)}=\frac{\varphi^{n-1}(t)}{\gamma(n)},
\end{equation*}
is increasing in $t$. This implies that $X\leq_{lr}U_{n(1)}$. Consequently, we can affirm that the assumption of Theorem \ref{thm:DREX} is satisfied, allowing us to conclude that $U_{n(1)}$ also possesses the DREX property.
\end{proof}
\begin{theorem}
If $U_{n(k_1)}$ is DREX, then $U_{n(k_2)}$ is also DREX when $k_1>k_2.$
\end{theorem}
\begin{proof}
Let $U_{n(k_j)},\ j=1,2,$ represent the $n$-th upper $k_j$-record values. The hazard rates of $U_{n(k_1)}$ and $U_{n(k_2)}$ are denoted as $\lambda_{n(k_1)}(t)$ and $\lambda_{n(k_2)}(t)$, respectively. It can be demonstrated that
\begin{equation*}
\lambda_{n(k_2)}(t)=\Pi(t)\lambda_{n(k_1)}(t),
\end{equation*}
where
\begin{equation*}\label{eq:PSI}
\Pi(t)=\left(\frac{k_2}{k_1}\right)^n\frac{\sum_{i=0}^{n-1} \frac{[-k_1\log \overline{F}(t)]^i}{i!}}{\sum_{i=0}^{n-1} \frac{[-k_1\log \overline{F}(t)]^i}{i!}},\ t>0.
\end{equation*}
Raqab and Amin \cite{raqab1997note} demonstrated that when $k_1 > k_2$, the function $\Pi(t)$ exhibits a monotonically increasing behavior with respect to $t.$ Conversely, it can be observed from \eqref{PDFkth} that the function
\begin{equation*}
\frac{f_{n(k_2)}(t)}{f_{n(k_1)}(t)}=\left(\frac{k_2}{k_1}\right)^{n}[\overline{F}(x)]^{k_2-k_1},\ t>0,
\end{equation*}
exhibits a monotonic increase with respect to $t$ when $k_1 > k_2$. This leads to the conclusion of the theorem.
\end{proof}

\section{Results on Past Extropy}
The initial theorem in this section establishes that the IPEX property of a stochastically larger RV in terms of the likelihood ratio order can be maintained by a smaller RV. This theorem serves as the foundation for deriving the subsequent conclusions.
\begin{theorem}\label{thm:IPEX}

Let $X\geq_{lr}Y$ and $\tau_Y(t)/\tau_X(t)$ be decreasing in $t\geq0.$ If $X$ is IPEX, then $Y$ is also IPEX.

\end{theorem}
\begin{proof}
Let $\delta(t)=\frac{\tau_Y(t)}{\tau_X(t)}$. By utilizing \eqref{eq:pex:3}, we can deduce that
\begin{eqnarray*}
  -4\widetilde{\mathfrak{J}}(Y;t)&=&\mathbb{E}[\tau_Y(Y_{22})|Y_{22}\leq t] \\
  &=&\mathbb{E}[\delta(Y_{22})\tau_X(Y_{22})|Y_{22}\leq t].
\end{eqnarray*}
Hence, the objective is to demonstrate that $\beta(t)=\mathbb{E}[\delta(Y_{22})\tau_X(Y_{22})|Y_{22}\leq t]$ is a decreasing function of $t\geq0$. After performing some manipulations, we obtain the following expression
\begin{eqnarray*}
  \beta'(t)&=&\left(\frac{\int_{0}^{t}\delta(y)\tau_X(y)2f_Y(x)F_Y(y)dy}
  {F_Y^2(t)}\right)' \\
  &=&2\tau_Y(t)[\delta(t)\tau_X(t)-\beta(t)].
\end{eqnarray*}
Therefore, it is sufficient to prove that $\beta(t)\geq\delta(t)\tau_X(t)$ for all $t\geq0$. For any $t\geq0$, we have the following inequality
\begin{eqnarray*}
  &&\beta(t)-\delta(t)\tau_X(t)\\
  &=&\int_{0}^{t}\frac{2f_Y(y)F_Y(y)}
  {F_Y^2(t)}[\delta(y)\tau_X(y)-\delta(t)\tau_X(t)]dy \\
  &=&\int_{0}^{t}\frac{2f_Y(y)F_Y(y)}
  {F_Y^2(t)}\delta(y)[\tau_X(y)-\tau_X(t)]dy \\
  &+&\int_{0}^{t}\frac{2f_Y(y)F_Y(y)}
  {F_Y^2(t)}\tau_X(t)[\delta(y)-\delta(t)]dy\\
  &=&A^\star_1(t)+A^\star_2(t).
\end{eqnarray*}
Let $\alpha^\star(t)=\mathbb{E}[\tau_X(X_{22})|X_{22}>t]$. According to the assumption, $\widetilde{\mathfrak{J}}(X;t)$ is decreasing for $t\geq0$. Consequently, $\alpha^\star(t)$ is increasing for $t\geq0$. It should be noted that $\alpha^{\star'}(t)=2\tau_X(t)(\tau_X(t)-\alpha^\star(t))$, and we have $\alpha^\star(t)=\tau_X(t)$ for all $t\geq0$. In other words,
\begin{equation}\label{extropy:12}
\int_{0}^{t}2f_X(x)S_X(x)[\tau_X(x)-\tau_X(t)]dx\geq0,
\end{equation}
for all $t\geq0.$ Furthermore, if $X\geq_{lr}Y$, it implies that the ratio $f_Y(y)/f_X(y)$ is a decreasing function for $y\geq0$. By applying Lemma 7.1(ii) from Barlow and Proschan \cite{barlow1975statistical} to \eqref{extropy:12}, we can directly deduce the following inequality for all $t\geq0$:
\begin{eqnarray*}
  A^\star_1(t)&=&\frac{1}{F_Y^2(t)}\int_{0}^{t}\frac{2f_Y(y)F_Y(y)\delta(y)}
  {2f(y)F_X(y)}2f_X(y)F_X(y)[\tau_X(y)-\tau_X(t)]dy \\
  &=&\frac{1}{F_Y^2(t)}\int_{0}^{t}\left(\frac{f_Y(y)}{f_X(y)}\right)^2
  2f_X(y)F_X(y)[\tau_X(y)-\tau_X(t)]dy\geq0.
\end{eqnarray*}
On the contrary, given that $\delta(x)$ is nonnegative and decreasing, we can conclude that for all $t\geq0$, the following inequality holds:
\begin{equation*}
  A^\star_2(t)=\int_{0}^{t}\frac{2f_Y(y)F_Y(y)}
  {F_Y^2(t)}\tau_X(t)[\delta(y)-\delta(t)]dy\geq0.
\end{equation*}
Based on the previous arguments, we can conclude that $\beta(t)\geq\delta(t)\tau_X(t)$ for all $t\geq0$.
\end{proof}
Let us consider the following useful theorem for coherent systems. Recall that the reversed hazard rate function of $T$ can be expressed as
\begin{equation}\label{hazard:system}
\tau_{T}(t)=\frac{f_T(t)}{F_T(t)}=\frac{\sum_{i=1}^{n}is_{i}{n\choose i}(\phi(t))^{i}}{\sum_{i=1}^{n}\left(\sum_{j=1}^{i}s_j\right){n\choose
i}(\phi(t))^i}\tau(t),
\end{equation}
where $\phi(t)=\frac{F(t)}{S(t)}$ is an increasing function of $t\in (0,\infty)$ and $\tau(t)=\frac{f(t)}{F(t)}$ represents the reversed hazard rate function of $X.$ It is clear that the function $\tau_{T}(t)/\tau(t)$ is increasing in $t$ if the rational function
\begin{equation}\label{rational2}
\widetilde{\psi}(x)=\frac{\sum_{i=1}^{n}is_{i}{n\choose i}x^{i}}{\sum_{i=1}^{n}\left(\sum_{j=1}^{i}s_j\right){n\choose
i}x^i},
\end{equation}
is increasing with respect to $x\in (0,\infty).$ We now present the next theorem, the proof of which is omitted as it follows a similar structure to that of Theorem \ref{thm:coherent}.
\begin{theorem}\label{thm:coherent2}
Under the condition of Theorem \ref{thm:coherent}, assume that $s_1\geq s_2\geq \ldots \geq s_n$, and the function $\widetilde{\psi}(x)$ defined in equation \eqref{rational2} is a decreasing function for $x>0$. If $X$ exhibits the IPEX property, then $T$ also possesses the IPEX property.
\end{theorem}
An immediate consequence of Theorem \ref{thm:coherent2} can be seen in the following.
\begin{theorem}
If $X$ is IPEX, then $X_{1:n}$ is also IPEX.
\end{theorem}
\begin{proof}
The signature of the series system is $\mathbf{s}=(1,0,\ldots,0)$, which satisfies the assumption of Theorem \ref{thm:coherent2}. Furthermore, it is easy to see that the function
\begin{equation*}
\widetilde{\psi}(x)=\frac{nx}{\sum_{i=1}^{n}{n\choose i}x^i},
\end{equation*}
is a decreasing function of $x$. This establishes the conclusion of the theorem.
\end{proof}
\begin{example}\rm
Consider the lifetime $T$ of a coherent system with signature $\textbf{s}=(\frac{1}{2},\frac{1}{2},0,0)$, consisting of $n=4$ i.i.d. components with a common CDF $F$. The system signature satisfies the assumption of Theorem \ref{thm:coherent2}. Furthermore, by examining \eqref{rational}, we can observe that
\begin{equation*}
\widetilde{\psi}(x)=\frac{6x^2+2x}{x^4+4x^3+6x^2+2x}.
\end{equation*}
\begin{figure}
\centering
\includegraphics[width=0.43\textwidth]{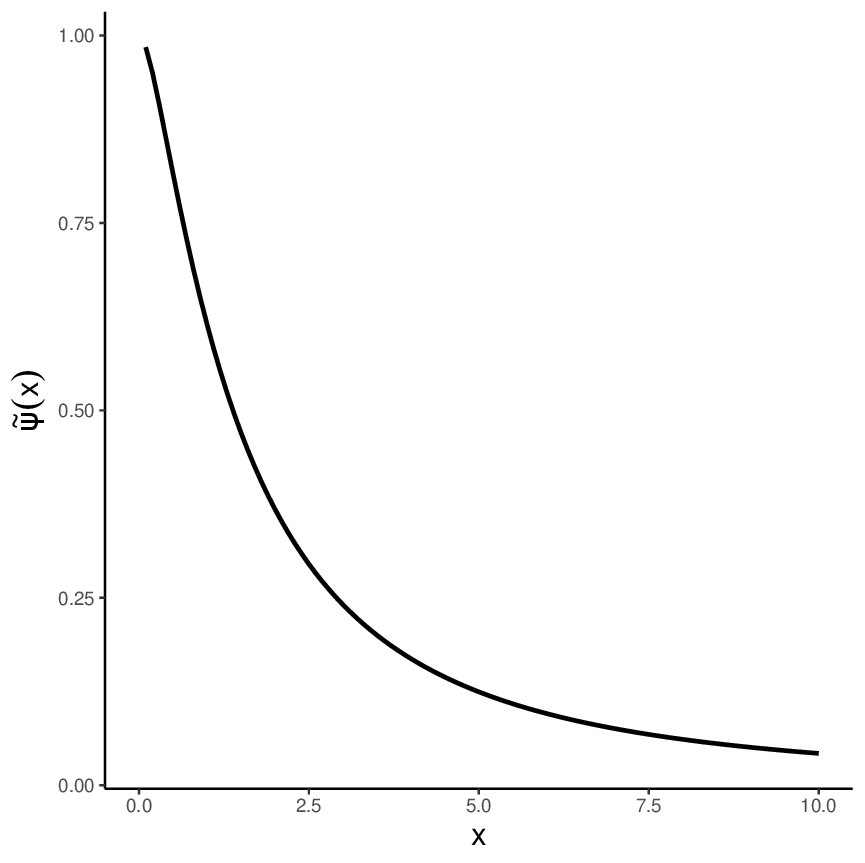}
\caption{The plot of the function $\widetilde{\psi}(x)$.} \label{fig:psitilde}
\end{figure}
Figure \ref{fig:psitilde} shows that the function $\widetilde{\psi}(x)$ is a decreasing function of $x.$ Therefore, by assuming the IPEX property of the component lifetime, we can conclude that the lifetime $T$ of the coherent system is also IPEX, as stated in Theorem \ref{thm:coherent2}.
\end{example}
%
%---------------------------------------------------------------------------------------
\section{Nonparametric Estimators}
%---------------------------------------------------------------------------------------
In this section, we introduce a nonparametric approach to estimating the residual extropy and past extropy as defined in Eqs. \eqref{def:extropy:time} and \eqref{eq:pex}. Let us assume we have a sequence of $n$ i.i.d. RVs ${X_i;1\leq i\leq n}$ with DF $f(x)$, CDF $F(x)$, and survival function $S(x)$. The kernel density estimator of the DF $f(x)$ was proposed by Silverman \cite{silverman2018density} as:
\begin{equation*}
f_n(x)=\frac{1}{nh_n}\sum_{i=1}^{n}K\left(\frac{x-X_i}{h_n}\right),\ x\in R,
\end{equation*}
where $h_n$ is a bandwidth or smoothing parameter and $K(\cdot)$ is a kernel function. The bandwidth sequence $h_n$ is chosen such that $h_n \to 0$ and $nh_n \to \infty$ as $n \to \infty$. This ensures that the kernel density estimator $f_n(x)$ converges to the true DF $f(x)$ as the sample size increases. The kernel function $K(\cdot)$ is a symmetric DF with finite variance. Some commonly used kernel functions include the normal (Gaussian), Epanechnikov, and tricube kernels. Hereafter, we employ the normal kernel function. With the kernel density estimator $f_n(x)$ in hand, we can now define nonparametric kernel-based estimators for the residual extropy $\mathfrak{J}(X;t)$ and the past extropy $\widetilde{\mathfrak{J}}(X;t)$ as:
\begin{equation}\label{estimate:resi}
\mathfrak{J}_n(X;t)=-\frac{1}{2}\int_{t}^{\infty}\left(\frac{f_n(u)}{S_n(t)}\right)^2du,
\end{equation}
\begin{equation}\label{estimate:past}
\widetilde{\mathfrak{J}}_n(X;t)=-\frac{1}{2}\int_{0}^t \left[\frac{f_n(u)}{F_n(t)}\right]^2du,
\end{equation}
where $F_n(t)=\int_{0}^{t}f_n(x)dx$ and $S_n(t)=\int_{t}^{\infty}f_n(x)dx.$  In the following, we outline the important properties of the nonparametric estimators ${f}_n(u)$ and ${S}_n(t)$ as follows:
\begin{equation}\label{Bias:f}
Bias(f_n(u))\simeq\frac{h_n^sc_s}{s!}f^{(s)}(u)\quad\text{and}\quad Var(f_n(u))\simeq\frac{C_k}{nh_n}f(u),
\end{equation}
and
\begin{equation}\label{Bias:S}
Bias(S_n(t))\simeq\frac{h_n^sc_s}{s!}\int_{t}^{\infty}f^{(s)}(u)du\quad\text{and}\quad Var(S_n(u))\simeq\frac{1}{nh_n}C_kS(t),
\end{equation}
where $c_s=\int_{-\infty}^{\infty}u^sK(u)du,\ C_k=\int_{-\infty}^{\infty}K^2(u)du$ and $f^{(s)}(u)$ is the $s$th derivative of $f$ with respect to $u.$ The next theorem investigates the consistency of the estimators defined in \eqref{estimate:resi} and \eqref{estimate:past}.
\begin{theorem}
The nonparametric kernel estimators $\mathfrak{J}_n(X;t)$ and $\widetilde{\mathfrak{J}}_n(X;t)$ are consistent estimators of $\mathfrak{J}(X;t)$ and $\widetilde{\mathfrak{J}}(X;t)$ respectively.
\end{theorem}
\begin{proof}
Applying the Taylor series approximation, we can derive the following expressions
\begin{eqnarray}
  \int_{t}^{\infty}f^2_n(u)du&\simeq&\int_{t}^{\infty}f^2(u)du+2\int_{t}^{\infty}f(u)(f_n(u)-f(u))du
  \label{eq:111} \\
  S^2_n(t)&\simeq&S^2(t)+2S(t)(S_n(t)-S(t))\label{eq:222}.
\end{eqnarray}
Employing the previously established results, we can now determine the bias and variance expressions for the nonparametric estimator $\int_{t}^{\infty}f^2_n(u)du$ given by
\[
Bias\left(\int_{t}^{\infty}f^2_n(u)du\right)\simeq2\int_{t}^{\infty}f(u)Bias(f_n(u))du,
\]
and
\[
Var\left(\int_{t}^{\infty}f^2_n(u)du\right)\simeq4\int_{t}^{\infty}f^2(u)Var(f_n(u))du.
\]
Building upon the preceding analysis, we can also characterize the bias and variance of the nonparametric survival function estimator $S_n^2(t)$ as
\[
Bias(S_n^2(t))\simeq 2S(t)Bias(S_n(t))\quad \text{and}\quad Var(S_n^2(t))\simeq4S^2(t)Var(S_n(t)).
\]
Incorporating the bias expressions from Eqs. \eqref{Bias:f} and \eqref{Bias:S}, and leveraging the asymptotic properties that $h_n \to 0$ and $nh_n \to \infty$ as $n \to \infty$, we can establish that the bias and variance of the nonparametric estimators $\int_{t}^{\infty}f^2_n(u)du$ and $S_n^2(t)$ both converge to 0 as the sample size $n$ approaches infinity. Consequently, as $n\rightarrow\infty,$ we can conclude that
\[
MSE\left(\int_{t}^{\infty}f^2_n(u)du\right)\rightarrow0\quad \text{and}\quad MSE(S_n^2(t))\rightarrow0.
\]
Therefore, $\int_{t}^{\infty}f^2_n(u)du\rightarrow\int_{t}^{\infty}f^2(u)du$ and $S_n^2(t)\rightarrow S^2(t)$, as $n\rightarrow\infty.$ Invoking Slutsky's theorem, we can obtain $\mathfrak{J}_n(X;t)\stackrel{p}{\rightarrow}\mathfrak{J}_n(X;t)$ as $n\rightarrow\infty.$ This means that $\mathfrak{J}_n(X;t)$ is a consistent estimator of $\mathfrak{J}(X;t)$. Similarly, we can show that the nonparametric estimator $\widetilde{\mathfrak{J}}_n(X;t)$  is a consistent estimator of the true past extropy $\widetilde{\mathfrak{J}}(X;t).$
\end{proof}
In the upcoming theorem, we establish that the estimators $\mathfrak{J}_n(X;t)$ and $\widetilde{\mathfrak{J}}_n(X;t)$ are asymptotically normally distributed under appropriate conditions.
\begin{theorem}
Let $\mathfrak{J}_n(X;t)$ and $\widetilde{\mathfrak{J}}_n(X;t)$ be the nonparametric kernel estimators of  $\mathfrak{J}(X;t)$ and $\widetilde{\mathfrak{J}}(X;t)$ respectively. Then for fixed $t,$ both
\[
\sqrt{nh_n}\left(\frac{\mathfrak{J}_n(X;t)-\mathfrak{J}(X;t)}{\sigma_{\mathfrak{J}_n}}\right)\quad \text{and}\quad \sqrt{nh_n}\left(\frac{\widetilde{\mathfrak{J}}_n(X;t)-\widetilde{\mathfrak{J}}(X;t)}{\sigma_{\widetilde{\mathfrak{J}}_n}}\right)
\]
has a standard normal distribution as $n\rightarrow\infty$ with
\[
\sigma^2_{\mathfrak{J}_n}=\frac{C_k}{S^4(t)}\left[\int_{t}^{\infty}f^3(x)dx+
\frac{\left(\int_{t}^{\infty}f^2(x)dx\right)^2}{S(t)}\right],
\]
and
\[
\sigma^2_{\widetilde{\mathfrak{J}}_n}=\frac{C_k}{F^4(t)}\left[\int_{0}^{t}f^3(x)dx+
\frac{\left(\int_{0}^{t}f^2(x)dx\right)^2}{F(t)}\right],
\]
where $C_k$ is as defined in (13).
\end{theorem}
\begin{proof}
It is not hard to see that
\begin{equation*}
\frac{\int_{t}^{\infty}f^2_n(u)du}{S_n^2(t)}-\frac{\int_{t}^{\infty}f^2(u)du}{S^2(t)}\simeq
\frac{S^2(t)\left(\int_{t}^{\infty}f^2_n(u)du-\int_{t}^{\infty}f^2(u)du\right)
-\int_{t}^{\infty}f^2(u)du\left(S_n^2(t)-S^2(t)\right)}{S^4(t)},
\end{equation*}
for all $t>0.$ Now, Eqs. \eqref{eq:111} and \eqref{eq:222} leads to
\begin{equation*}
\mathfrak{J}_n(X;t)-\mathfrak{J}(X;t)\simeq -\frac{1}{S^2(t)}\int_{t}^{\infty}f(u)(f_n(u)-f(u))du+
\frac{\int_{t}^{\infty}f^2(u)du}{S^4(t)}(S_n(t)-S(t))S(t).
\end{equation*}
The proof is then completed using the asymptotic normality of $f_n(u)$ given by Roussas \cite{roussas1990asymptotic}. Analogously, the asymptotic normality of the past extropy estimator $\widetilde{\mathfrak{J}}_n(X;t)$  can also be established using similar techniques.
\end{proof}
\subsection{Simulation studies}
To thoroughly examine the behavior and performance characteristics of the proposed nonparametric extropy estimators, we conduct a Monte Carlo simulation study. For evaluating the kernel-based estimator of the REX, we utilize the standard exponential distribution as the underlying data-generating process. Conversely, to assess the estimator of the PEX, we employ the standard uniform distribution. In this case, it is not hard to see that $\mathfrak{J}(X;t)=-{1}/{4}$ and $\widetilde{\mathfrak{J}}(X;t)=-{1}/{2t}$ for all $t>0.$ The bias and root mean squared error (RMSE) are computed for various sample sizes $n=40, 50, 100,$ different values of $t$ and the bandwidth $h_n.$ A total of $5000$ iterations are conducted, and the results are presented in Tables \ref{weibull1}-\ref{weibull2}. It is clear that the RMSE decreases as
the sample size increases.
%
%%%%%%%%%%%%%%%%%%%%%%%%%%%%%%%%%%%%%%%%%%%%%%%%%%%%%%
\begin{table}[h] \centering
\caption{\label{weibull1} The Bias and RMSE of the estimate of the REX for different choices of $t,n$ and
$h_n.$}
%%%%%%%%%%%%%%%%%%%%%%%%%%%%%%%%%%%%%%%%%%
\vspace{0.30cm}
\resizebox{18cm}{!}{
%%%%%%%%%%%%%%%%%%%%%%%%%%%%%%%%%%%%%%%%%%
\begin{tabular}{cccccccccccccccccc}
    \hline
    \hline
&&\multicolumn{2}{c}{$h_n=0.1$}&\multicolumn{2}{c}{$h_n=0.3$}& \multicolumn{2}{c}{$h_n=0.5$}&\multicolumn{2}{c}{$h_n=0.7$}&\multicolumn{2}{c}{$h_n=0.9$}& \\
\hline
$n$ & $t$ &Bias&RMSE&Bias&RMSE&Bias&RMSE&Bias&RMSE&Bias&RMSE&\\
\hline
40 & 0.1& -0.038085&0.059534& 0.010660&0.037561& 0.037844&0.046405& 0.055946&0.061081& 0.070453&0.074467\\
   & 0.3& -0.051045&0.074785&-0.009041&0.043960& 0.018119&0.038829& 0.040444&0.049441& 0.057866&0.063829\\
   & 0.5& -0.064114&0.088706&-0.023851&0.054890& 0.000843&0.040533& 0.024819&0.042633& 0.042737&0.054026\\
   & 0.7& -0.077554&0.106174&-0.033854&0.065484&-0.015704&0.051362& 0.006712&0.042175& 0.026815&0.047245\\
   & 0.9& -0.097257&0.131391&-0.044237&0.079279&-0.028607&0.063402&-0.009240&0.052396& 0.011515&0.051262 \\
50 & 0.1& -0.030202&0.050679& 0.014239&0.034098& 0.039855&0.047251& 0.059226&0.062591& 0.074376&0.076681 \\
   & 0.3& -0.041059&0.061781&-0.005135&0.038446& 0.022257&0.037052& 0.044294&0.050211& 0.060934&0.065228 \\
   & 0.5& -0.049871&0.073093&-0.017570&0.047797& 0.006085&0.035809& 0.029365&0.041083& 0.047878&0.054740\\
   & 0.7& -0.063423&0.089724&-0.025864&0.059020&-0.009543&0.042997& 0.012594&0.036923& 0.033401&0.046125\\
   & 0.9& -0.076567&0.103604&-0.033456&0.066549&-0.020781&0.053567&-0.001605&0.042017& 0.018159&0.041508\\
100& 0.1& -0.012518&0.031112& 0.021881&0.030522& 0.045567&0.048521& 0.064972&0.066576&0.080664&0.081523\\
   & 0.3& -0.019357&0.038169& 0.005014&0.026334& 0.029231&0.035302& 0.050825&0.053541&0.069328&0.070452\\
   & 0.5& -0.024349&0.044137&-0.006644&0.030535& 0.014082&0.027651& 0.037193&0.041856&0.057059&0.058911\\
   & 0.7& -0.030529&0.051764&-0.012753&0.035958& 0.002491&0.027085& 0.023840&0.031864&0.044343&0.048285\\
   & 0.9& -0.037869&0.059713&-0.016207&0.040609&-0.006254&0.032069& 0.011541&0.027662&0.031952&0.038681\\
\end{tabular}
}
\end{table}
%%%%%%%%%%%%%%%%%%%%%%%%%%%%%%%%%%%%%%%%%%%%%%%%
%
%%%%%%%%%%%%%%%%%%%%%%%%%%%%%%%%%%%%%%%%%%%%%%%%%%%%%%
\begin{table}[h] \centering
\caption{\label{weibull2} The Bias and RMSE of the estimate of the PEX for different choices of $t,n$ and
$h_n.$}
%%%%%%%%%%%%%%%%%%%%%%%%%%%%%%%%%%%%%%%%%%
\vspace{0.30cm}
\resizebox{18cm}{!}{
%%%%%%%%%%%%%%%%%%%%%%%%%%%%%%%%%%%%%%%%%%
\begin{tabular}{cccccccccccccccccc}
    \hline
    \hline
&&\multicolumn{2}{c}{$h_n=0.1$}&\multicolumn{2}{c}{$h_n=0.3$}& \multicolumn{2}{c}{$h_n=0.5$}&\multicolumn{2}{c}{$h_n=0.7$}&\multicolumn{2}{c}{$h_n=0.9$}& \\
\hline
$n$ & $t$ &Bias&RMSE&Bias&RMSE&Bias&RMSE&Bias&RMSE&Bias&RMSE&\\
\hline
40 & 0.3& -0.224990&0.302406&-0.195709&0.283181&-0.175910&0.277147&-0.164311&0.283348&-0.164159&0.258154\\
   & 0.5& -0.100627&0.119046&-0.079870&0.100044&-0.063865&0.088384&-0.056714&0.084291&-0.055578&0.083333\\
   & 0.7& -0.063200&0.075474&-0.042561&0.052825&-0.032796&0.042699&-0.028561&0.040864&-0.027497&0.039973\\
   & 0.9& -0.046883&0.054922&-0.028877&0.033568&-0.020798&0.026526&-0.018064&0.023694&-0.016682&0.023635\\
   & 1.2& -0.093213&0.097936&-0.042628&0.045015&-0.021390&0.023654&-0.013744&0.016603&-0.011278&0.014185\\
50 & 0.3& -0.182023&0.237535&-0.157348&0.213388&-0.136926&0.196992&-0.125264&0.188979&-0.129408&0.195837\\
   & 0.5& -0.082346&0.101558&-0.068042&0.083678&-0.053127&0.069245&-0.046588&0.065578&-0.045154&0.063562\\
   & 0.7& -0.052339&0.061625&-0.037929&0.044500&-0.027700&0.034762&-0.023553&0.032967&-0.022152&0.031334\\
   & 0.9& -0.038676&0.044352&-0.025145&0.028471&-0.017438&0.021324&-0.014259&0.018992&-0.013384&0.018296\\
   & 1.2& -0.085876&0.089546&-0.039997&0.041699&-0.019487&0.020997&-0.011701&0.014000&-0.009328&0.011771\\
100& 0.3& -0.109826&0.133963&-0.092358&0.109498&-0.071569&0.094341&-0.064703&0.089040&-0.061479&0.089410\\
   & 0.5& -0.050749&0.059613&-0.046531&0.052216&-0.030685&0.037400&-0.024295&0.031799&-0.022137&0.031806\\
   & 0.7& -0.031093&0.035881&-0.026305&0.028974&-0.016906&0.019890&-0.012794&0.016643&-0.011327&0.015405\\
   & 0.9& -0.022528&0.025497&-0.018178&0.019492&-0.010785&0.012548&-0.008189&0.010128&-0.006920&0.009528\\
   & 1.2& -0.072761&0.074239&-0.035175&0.035777&-0.015625&0.016053&-0.008168&0.008871&-0.005505&0.006554\\
\end{tabular}
}
\end{table}
%%%%%%%%%%%%%%%%%%%%%%%%%%%%%%%%%%%%%%%%%%%%%%%%
%
\subsection{A real data}
To study the COVID-19 pandemic spread, Kasilingam \emph{et al.} \cite{kasilingam2021exploring} employed an exponential model. Their analysis targeted nations showing initial indications of containment efforts until March 26, 2020. The dataset comprised infected case percentages from 40 countries (excluding zero values), as listed below:\\
\noindent Data Set: 1.56, 8.51, 2.17, 0.37, 1.09, 9.84, 4.95, 3.18, 11.37, 2.81, 6.22, 1.87, 9.05, 2.44, 1.38, 4.17, 3.74, 1.37, 2.33, 7.80, 2.10, 0.47, 2.54, 4.92, 0.09, 0.18, 1.72, 1.02, 0.62, 2.34, 0.50, 2.37, 3.65, 0.59, 5.76, 2.14, 0.88, 0.95, 4.17, 2.25.\\
%
%%%%%%%%%%%%%%%%%%%%%%%%%%%%%%%%%%%%%%%%%%%%%%%%%%%%%%
\begin{table}[h] \centering
\caption{\label{estiamte:Teheo} Comparison of theoretical values and estimates of REX based on the exponential distribution for genuine COVID-19 infection data.}
%%%%%%%%%%%%%%%%%%%%%%%%%%%%%%%%%%%%%%%%%%
\vspace{0.30cm}
\resizebox{18cm}{!}{
%%%%%%%%%%%%%%%%%%%%%%%%%%%%%%%%%%%%%%%%%%
\begin{tabular}{cccccccccccccccccc}
    \hline
    \hline
&\multicolumn{2}{c}{$h_n=0.1$}&\multicolumn{2}{c}{$h_n=0.3$}& \multicolumn{2}{c}{$h_n=0.5$}&\multicolumn{2}{c}{$h_n=0.7$}&\multicolumn{2}{c}{$h_n=0.9$}& \\
\hline
$t$ &$\mathfrak{J}(X;t)$&$\mathfrak{J}_n(X;t)$&$\mathfrak{J}(X;t)$&$\mathfrak{J}_n(X;t)$&$\mathfrak{J}(X;t)$&
$\mathfrak{J}_n(X;t)$&$\mathfrak{J}(X;t)$&$\mathfrak{J}_n(X;t)$&$\mathfrak{J}(X;t)$&$\mathfrak{J}_n(X;t)$&\\
\hline
 0.1& -0.0800&-0.1090&-0.0800&-0.0854&-0.0800&-0.0799&-0.0800&-0.0768&-0.0800&-0.0741\\
 0.3& -0.0800&-0.1139&-0.0800&-0.0884&-0.0800&-0.0825&-0.0800&-0.0793&-0.0800&-0.0764\\
 0.5& -0.0800&-0.1185&-0.0800&-0.0913&-0.0800&-0.0852&-0.0800&-0.0818&-0.0800&-0.0787\\
 0.7& -0.0800&-0.1227&-0.0800&-0.0939&-0.0800&-0.0878&-0.0800&-0.0844&-0.0800&-0.0811\\
 0.9& -0.0800&-0.1287&-0.0800&-0.0967&-0.0800&-0.0903&-0.0800&-0.0869&-0.0800&-0.0835 \\
\end{tabular}
}
\end{table}
%%%%%%%%%%%%%%%%%%%%%%%%%%%%%%%%%%%%%%%%%%%%%%%%
%
The maximum likelihood estimate for the exponential distribution parameter is $\widehat{\lambda}=0.32$ and hence we examine the proximity of the REX estimator to the theoretical REX value. Various combinations of $t$ and $h_n$ are detailed in Table \ref{estiamte:Teheo}. Analysis of Table \ref{estiamte:Teheo} reveals that REX estimates approach
the theoretical value as $h_n$ and $t$ increases.

%---------------------------------------------------------------------------------------
\section{Conclusion}
This study has presented new findings on the residual and past extropy of an RV $X$. In the context of reliability theory, stochastic comparisons between $X$ and its residual lifetime play a crucial role in analyzing aging properties. We have performed a characterization of the DREX property in this domain using the monotonicity properties of the REX. Furthermore, we identified sufficient conditions for the preservation of the DREX and IPEX properties. These results have been applied to coherent systems, order statistics, and record values, providing valuable insights for researchers and practitioners engaged in extropy-based analysis. The results of this study contribute to a deeper understanding of the behavior and properties of RVs and provide a solid foundation for further research in this area.
Finally, we have introduced nonparametric estimators for the residual and past extropy measures. To evaluate the practical performance of these estimators, we have conducted Monte Carlo simulations and further included a real data analysis. The simulation studies have demonstrated the strong performance of the nonparametric extropy estimators. They have exhibited desirable statistical properties, such as low bias and RMSE, even in finite sample scenarios. The outcomes of this study can have potential applications in lithium-ion batteries remaining useful life prediction, leveraging extropy feature extraction and support vector regression techniques, as explored in similar work by Jia \emph{et al.} \cite{jia2021sample}. The proposed methods can be potentially applied in the field of remaining useful life prediction for rolling bearings, as suggested by Zhang \emph{et al.} \cite{zhang2023remaining}. The concepts and methodologies discussed in this paper have practical implications for machine learning and deep learning, benefiting from related research such as  Al-Qazzaz \emph{et al.} \cite{al2023automatic}, Wu \emph{et al.} \cite{wu2022enhance}, and Liu \emph{et al.} \cite{liu2018connectionist}, among others. By building upon the insights and techniques from these publications, researchers and practitioners can further enhance their models and algorithms in these domains.
\bigskip

%---------------------------------------------------------------------------------------------------------------
\noindent \textbf{Acknowledgments}\medskip

\noindent This work was supported by Researchers Supporting Project number (RSP2024R392), King Saud University,
Riyadh, Saudi Arabia.\textbf{\medskip }

%%%%%%%%%%%%%%%%%%%%%%%%%

\end{document}